\documentclass{article}

\usepackage{graphics}
\usepackage{amssymb}

\usepackage{palatino,xcolor}
\usepackage{amsfonts,amsmath} 
\usepackage{amsthm}
\usepackage{graphicx,enumerate}
\usepackage[hmarginratio=1:1,hscale=0.75]{geometry}
\usepackage[latin1]{inputenc}
\usepackage[all]{xy}

\usepackage{amssymb,mathtools}
\usepackage{stmaryrd}
\usepackage{hyperref}
\usepackage{pdfsync}
\usepackage{color}
\usepackage[T1]{fontenc} 
\usepackage{lmodern}

\usepackage{xcolor}
\usepackage{soul}

\setstcolor{purple}

 \newtheorem{theorem}{Theorem}[section]
\newtheorem{lemma}{Lemma}[section]

\newtheorem{corollary}[lemma]{Corollary}

\newtheorem{remark}[lemma]{Remark}

\newcommand{\R}{\mathbb{R}}

\renewcommand{\H}{\mathcal{H}} 

\newcommand{\E}{E_{p,\la}} 
 
\newcommand{\A}{\mathcal{A}}

\newcommand{\lra}{\longrightarrow} 
\newcommand{\Lra}{\Longrightarrow}

 \newcommand{\sse}{\subseteq}
 
\renewcommand{\-}{\bar}

\newcommand{\pd}{\partial} 
\newcommand{\rhu}{\;{\overset{*}{\rightharpoonup}}\; }

\newcommand{\8}{\infty} 

\newcommand{\vph}{\varphi}
\newcommand{\vep}{\varepsilon} 

\newcommand{\al}{\alpha}
 \newcommand{\bt}{\beta}
\newcommand{\gm}{\gamma}
 
\newcommand{\sm}{\Sigma}
 \newcommand{\om}{\Omega}
\newcommand{\la}{\lambda} 
\renewcommand{\vartheta}{\Theta}

\DeclareMathOperator{\dist}{dist}

\DeclareMathOperator{\diam}{{\operatorname{diam}}}

{   \end{list} }

\renewcommand{\d}{\,{\operatorname{d}}}

\title{The average distance problem with perimeter-to-area ratio penalization\footnote{This paper will appear in SIAM Journal on Mathematical Analysis. }}

\author{Qiang Du\thanks{Department of Applied Physics and Applied Mathematics, and Data Science Institute, Columbia University,
500 W. 120th St., New York, NY 10027, USA. Email: qd2125@columbia.edu}
\and Xin Yang Lu \thanks{Department of Mathematical Sciences, Lakehead University, 955 Oliver Rd., Thunder Bay, ON
P7B 5E1, Canada 
AND
Department of Mathematics and Statistics, McGill University,
805 Sherbrooke St. W., Montreal, QC H3A 0B9, Canada. Email: xlu8@lakeheadu.ca}
\and Chong Wang \thanks{Department of Mathematics, Washington and Lee University,
204 W Washington St., Lexington, VA
24450, USA. Email: cwang@wlu.edu}}
\date{}

\begin{document}
\maketitle

\begin{abstract}
\noindent 
In this paper we consider the functional
\begin{equation*}
E_{p,\la}(\Omega):=\int_\Omega \dist^p(x,\pd \Omega )\d x+\la \frac{\H^1(\pd \Omega)}{\H^2(\Omega)}.
\end{equation*}
Here 
$p\geq 1$, $\la>0$ are given parameters, the unknown $\Omega$ varies among compact, convex, Hausdorff two-dimensional sets of $\R^2$, $\pd \Omega$
denotes the boundary of $\Omega$, and $\dist(x,\pd \Omega):=\inf_{y\in\pd \Omega}|x-y|$. The integral term 
$\int_\Omega \dist^p(x,\pd \Omega )\d x$ quantifies the ``easiness'' for points in $\Omega$ to reach the boundary,
while $\frac{\H^1(\pd \Omega)}{\H^2(\Omega)}$ is the perimeter-to-area ratio.
The main aim is to prove existence and
$C^{1,1}$-regularity of minimizers of $\E$.
 \end{abstract}
 
\textbf{Keywords.}
perimeter-to-area ratio, regularity

\textbf{Classification. }
49Q20, 
 49K10, 
  35B65 

\section{Introduction}

The perimeter-to-area ratio (in 2D), or surface area-to-volume ratio (in 3D), plays a crucial role in many 
processes. In biology, for instance, the size of prokaryote cells is limited by the efficiency of diffusion processes,
 fundamental to transport nutrients across the cell, which is strongly correlated with the surface area-to-volume ratio.
A larger surface area-to-volume ratio also gives prokaryote cells a high metabolic rate, fast growth, and short lifespan
compared to eukaryote cells (see for instance \cite{pro}).

 
  In chemistry, higher surface area-to-volume ratio increases the typical speed
  of chemical reactions. This phenomenon can be observed in many
 instances , sometimes quite dramatically, 
such as dust explosions, when dust particles of seemingly non-flammable materials (e.g., aluminum, sugar, flour, etc.) can be ignited
 due to their very large surface area-to-volume ratio (\cite{dust,dust-book}).
 
 \medskip
 
 In this paper we will focus on the 2D case.
In the above examples, there are essentially two
often competing quantities: one is the ``easiness'' to access the boundary, and the other
is the perimeter-to-area ratio.

A very thin, rod-like, rectangular body would have very good access to boundary (desirable), but large perimeter-to-area ratio. A disk 
 would have the lowest perimeter-to-area ratio (desirable) among
shapes of the same total area, but access to boundary would be limited.
It is also possible to have both large perimeter-to-area ratio and
limited access to boundary. 


 Until now, we have discussed the ``easiness'' of accessing the boundary only at a qualitative
 level.
In order to quantify it, we introduce the ``average distance'' term
\begin{equation*}
F_p(\Omega):=
\int_{\Omega} \dist^p(x,\pd \Omega)\d x,
\end{equation*}
where $\dist(x,\pd \Omega):=\inf_{y\in \pd \Omega}|x-y|$; $p\geq 1$ is a given parameter; and $|\cdot|$ denotes the Euclidean distance.


Consider the energy functional
\begin{eqnarray} \label{func}
E_{p,\lambda} (\Omega) = \int_{\Omega}  \text{ dist}^p (x, \partial \Omega) \ dx + \lambda \frac{ {\mathcal H}^1 (\partial \Omega)}{{\mathcal H}^2 ( \Omega)},
\end{eqnarray}
where $p \geq 1, \lambda > 0$ are given parameters.
Define the admissible set
\begin{eqnarray}
{\mathcal {A}}  := \{ \Omega :  \Omega \subset \mathbb{R}^2  \text{ is compact, convex and Hausdorff two-dimensional} \}. \nonumber
\end{eqnarray}
The term $\frac{\H^1(\pd \Omega)}{\H^2(\Omega)}$ is the perimeter-to-area ratio.
Note that neither the perimeter $\H^1(\pd \Omega)$, nor the area $\H^2(\Omega)$, is penalized, only their ratio is. This 
makes compactness results quite challenging to prove, and several estimates (in Section \ref{prel})
will be required. Another issue is that it is not very clear if the average-distance term is
just a lower order perturbation of $\frac{\H^1(\pd \Omega)}{\H^2(\Omega)}$.
The role of convexity is to ensure crucial compactness estimates (Lemmas \ref{p}
and \ref{exist}).
Note that $E_{p,\la}$ is invariant under rigid movements. Further details about the space of convex sets,
and its topology, will be discussed in Section \ref{prel}.
The main result of this paper is:
\begin{theorem}\label{main}
 Given $p \geq 1, \lambda > 0$, the following assertions hold:
 \begin{itemize}
 \item[(1)]  $E_{p,\lambda}$ admits a minimizer in ${\mathcal {A}}$.
 \item[(2)]   All minimizers are compact, convex, $C^{1,1}$-regular sets, with Hausdorff dimension equal to $2$.
 \item[(3)]   The perimeter-to-area ratio of any minimizer $\Omega$ satisfies
             $$\frac{\H^1(\pd \Omega)}{\H^2(\Omega)} = \frac{p+2}{\la(p+3)}\min_{\mathcal A} \E.$$
 \end{itemize}
\end{theorem}
Here, and for future reference, the expression ``$\Omega$ is $C^{k}$-regular'' means that its boundary
$\pd \Omega$ is $C^k$-regular, i.e., $\pd \Omega$ admits
a $C^k$-regular parameterization. 

\medskip

Note that the functional $F_p$ is formally similar to the average-distance functional
\begin{equation*}
\Sigma\mapsto \int_{\Gamma}\dist^p(x,\Sigma)\d\mu,
\end{equation*}
where $\Gamma$ is a given domain, $\mu$ a given measure on $\Gamma$, and $\Sigma$ varies among compact,
path-wise connected sets with Hausdorff dimension equal to 1.
The average-distance functional has been
widely studied, and used in several modeling problems. For a (non exhaustive) list 
of references, we cite the papers (and books) by Buttazzo and collaborators \cite{BMS,BOS, BPSS, BS3,BS4,BS1,BS2}.
 Also related are the papers by Paolini and Stepanov \cite{PaoSte},
 Santambrogio and Tilli \cite{ST}, Tilli \cite{Til}, Lemenant and Mainini \cite{LM},
 Slep\v{c}ev \cite{Slepcev}, and the review paper by Lemenant \cite{Lem}.
 Similar variational problems entailing a competition between classical perimeter and nonlocal repulsive interaction were studied by Muratov and Kn\"upfer \cite{mk},
 Goldman, Novaga and Ruffini \cite{gnru}, and Goldman, Novaga and R\"oger \cite{gnro}.  Figalli, Fusco, Maggi, Millot, and Morrini studied a competition between a nonlocal $s$-perimeter and a nonlocal repulsive interaction term \cite{ffmmm}.


The rest of the paper is structured as follows: section \ref{prel} is dedicated to proving some auxiliary estimates on the area (\eqref{a}
and Corollary \ref{a1}) and perimeter (Lemma \ref{p}) of elements of minimizing sequences. Existence of minimizers will be shown in section \ref{sec:existence}, while $C^{1,1}$ regularity will be proven in section \ref{sec:proof}. 
Finally, we explore several future directions to further our understanding of the penalized average distance problem.

\section{Preliminary estimates} \label{prel}

In this section we collect some preliminary estimates that will be used later.
First, we remark that
given $p\geq 1$ and $\la>0$, for any $ \Omega \in {\mathcal A}$ it holds
\begin{equation}
 \H^2(\Omega) \geq \frac{4\pi\la^2}{\E(\Omega)^2}.
 \label{a}
\end{equation}
Indeed, 
consider an arbitrary $ \Omega \in {\mathcal A}$. By the isoperimetric inequality, among all convex sets
with area $\H^2(\Omega)$, the perimeter-to-area ratio is minimum for a disk,
where it attains the value $2\sqrt{\pi}/\sqrt{\H^2(\Omega)}$.
Hence
\begin{equation*}
\frac{2\la\sqrt{\pi}}{\sqrt{\H^2(\Omega)}} \leq \la\frac{\H^1(\pd \Omega)}{\H^2(\Omega)} \leq \E(\Omega),
\end{equation*}
and \eqref{a} is proven.

\begin{corollary} \label{a1}
Given $p\geq 1$, $\la>0$, any minimizing sequence $\Omega_n\sse {\mathcal A}$ satisfies
\begin{align}
\H^2(\Omega_n) &\geq 4\pi\la^2 \bigg( \frac{2\pi}{p^2+3p+2}+2\la+1  \bigg)^{-2} =:C_1, \label{a1-1}\\
\frac{\H^1(\pd \Omega_n)}{\H^2(\Omega_n)} &\leq \frac1\la\bigg( \frac{2\pi}{p^2+3p+2}+2\la+1  \bigg)=: C_2,\label{a1-2}
\end{align}
for any sufficiently large $n$.
\end{corollary}

\begin{proof}
First we prove $\inf_{\mathcal A} \E<+\8$. Let $B_1\in {\mathcal A}$ be a disk of radius $1$. Direct computation gives
\begin{align}
\inf_{{\mathcal A}} \E\leq \E(B_1)&=\int_{B_1} \dist^p(x,\pd B_1)\d x+ \la \frac{\H^1(\pd B_1)}{\H^2(B_1)}\notag\\
&=2\pi\int_0^1 (1-r)^pr\d r +2\la
=\frac{2\pi}{p^2+3p+2}+2\la<+\8.\label{inf E}
\end{align}
Thus, given a minimizing sequence $\Omega_n\sse {\mathcal A}$, there exists $N$ such that for any $n\geq N$ it holds
\begin{equation}
\E(\Omega_n) \leq \frac{2\pi}{p^2+3p+2}+2\la+1, \label{a1 E}
\end{equation}
and \eqref{a} gives
\begin{equation*}
  \H^2(\Omega_n) \geq 4\pi\la^2 \bigg( \frac{2\pi}{p^2+3p+2}+2\la+1  \bigg)^{-2},
\end{equation*}
for any $n\geq N$, hence \eqref{a1-1}. To prove \eqref{a1-2}, note that \eqref{a1 E} forces
\begin{equation*}
\frac{2\pi}{p^2+3p+2}+2\la+1\geq \E(\Omega_n)\geq \la\frac{\H^1(\pd \Omega_n)}{\H^2(\Omega_n)},
\end{equation*}
concluding the proof.
\end{proof}


\begin{lemma}\label{p}
Given $p\geq 1$ and $\la>0$, for any minimizing sequence $\Omega_n\sse {\mathcal A}$,
it holds, for all sufficiently large $n$,
\begin{equation}\label{p-1}
 \H^1(\pd \Omega_n) \leq  C_3 = C_3(p,\la)
\end{equation}
 with  $C_3$ being some computable (but uninfluential) constant.
\end{lemma}

\begin{proof} We first claim that
for any $\Omega \in {\mathcal A}$ it holds
\begin{equation}\label{p claim1}
\int_{\Omega} \dist^p(x,\pd \Omega)\d x \geq C \frac{\H^2(\Omega)^{p+1}}{\H^1(\pd \Omega)^{p}},
\qquad C= 3^{-p}   2^{-p-4} .
\end{equation}

 Consider an arbitrary $\Omega \in {\mathcal A}$.  Let $A,B\in \pd \om$ be two points realizing
 $D:=|A-B| =\diam \om $. Let $\sm_i$, $i=1,2$ be the 
 lines (see Figure \ref{construction upper bound on perimeter}) orthogonal to
  the line segment
 between $A$ and $B$ (which we denote by $ \llbracket A,B \rrbracket$).
  Since $\om$ is convex, and $|A-B| =\diam \om $, $\om$
 is entirely contained in the region between $\sm_1$ and $\sm_2$.
Then let $P_i$, $i=1,2$ be the points on $\pd\om$ such that 
the triangles $\triangle A P_i B$ have maximal areas. As $\om$ is convex, we have
\[ \H^2( \om) \le D(h_1+h_2),\qquad h_i := \dist( P_i, \llbracket A,B \rrbracket ).\]
On the other hand, $\H^2(\triangle A P_i B) = Dh_i/2$, hence
\[ \frac{\H^2(\triangle A P_1 B \cup \triangle A P_2 B )}{\H^2(\om)} \ge \frac{1}{2}. \]

 \begin{figure}[ht]
 \centering
 	  	\includegraphics[scale=0.6]{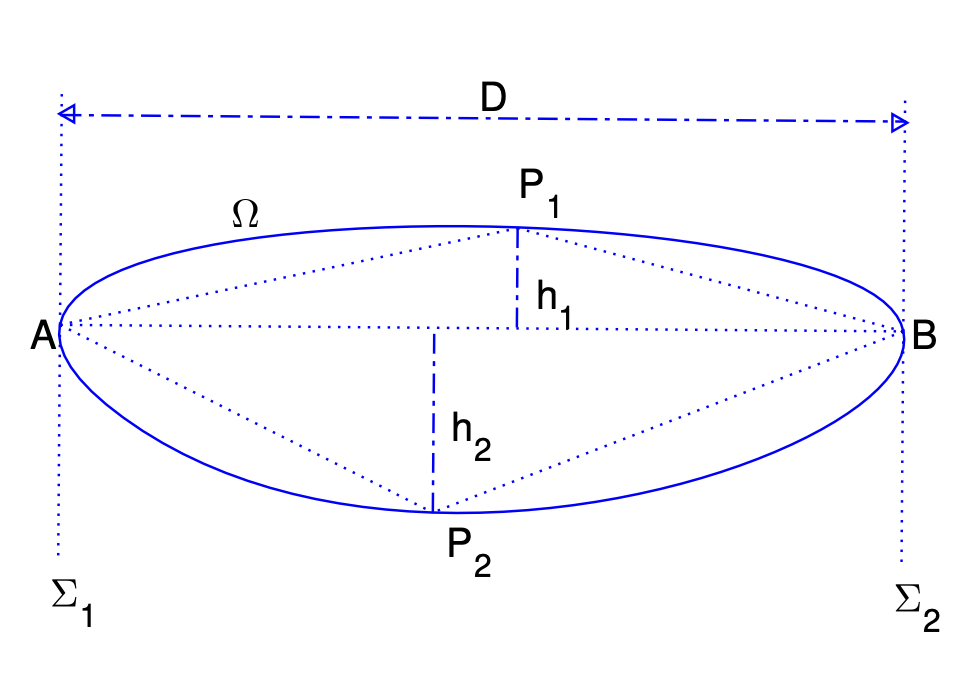}
 	\caption{ A schematic representation of the construction. The points $P_1$ and
 	(resp.  $P_2$) are the points on $\pd\om$ above (resp. below) 
 	the segment $\llbracket A,B \rrbracket$ furthest away from $\llbracket A,B \rrbracket$. }
 	\label{construction upper bound on perimeter}
 \end{figure}

Now we do the following construction: let $O_i$
(resp. $r_i$)
 be the incenter (resp. inradius) of $\triangle A P_i B $, $i=1,2$. Denote by 
 $\~A$ (resp. $\~P_i$, $\~B$) the midpoints of the line segments between
 $O_i$ and $A$ (resp. $P_i$, $B$) -- see Figure \ref{construction lower bound on average-distance}.

  \begin{figure}[ht]
  \centering
  	  	\includegraphics[scale=.6]{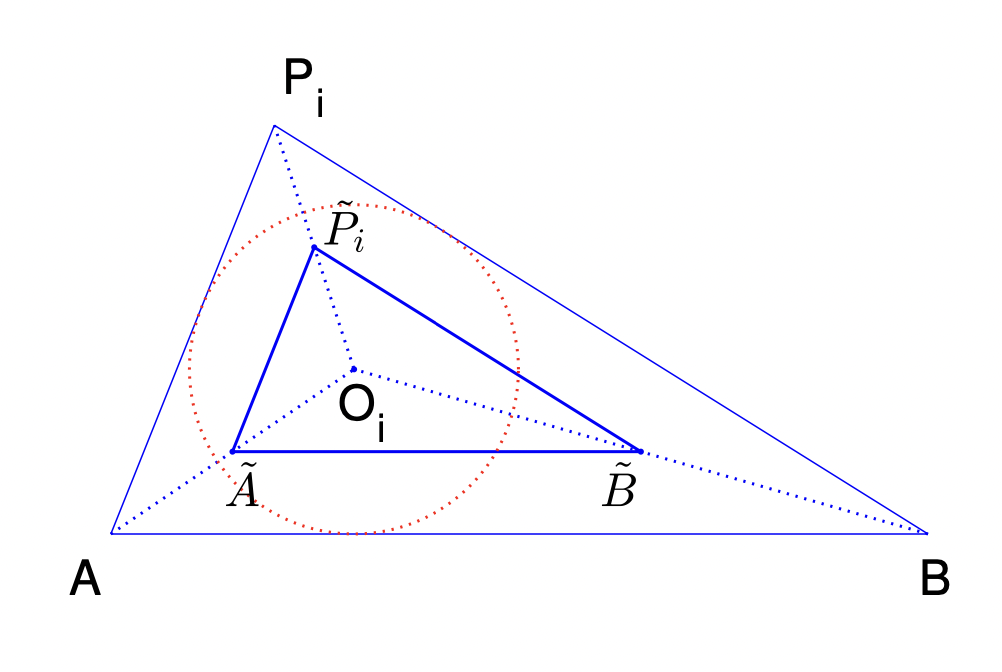}
 	\caption{A schematic representation of the construction. The points $\~A$, $\~P_i$,
 	$\~B$ are the midpoints of the segments $\llbracket O_i,A \rrbracket$,
 $\llbracket O_i,P_i \rrbracket$, $\llbracket O_i,B \rrbracket$ respectively. The red dotted circle
is the incircle of the triangle $\triangle AP_iB$.}
 	\label{construction lower bound on average-distance}
 \end{figure}

  Clearly, $\triangle \~A \~P_i\~B$ is a rescaled copy
 of $\triangle A P_i B $, with area $\H^2(\triangle A P_i B )/4$.
As 
\[ \text{inradius} = \frac{2 \text{Area}}{\text{Perimeter}}, \]
we can estimate $r_i$ as follows:
 \begin{align}
 r_i & = \frac{D h_i}{ D + |A-P_i|+|B-P_i| } \ge \frac{Dh_i}{3D} = \frac{h_i}{3}
 \label{inradius lower bound}
 \end{align}
 since by definition we have $D=\diam \om \ge |A-P_i|, |B-P_i| $.
  Then, noting that 
  \[  \dist(x,\pd \om) \ge \dist(x,\pd A P_i B) \ge \frac{1}{2} \dist(O_i, \pd A P_i B) \ge \frac{1}{2} r_i \]
 for all $x\in  \triangle \~A \~P_i\~B, \; i=1,2, $ 
  we have
 \begin{align}
 \int_\om \dist^p ( x,\pd\om )\d x&\ge 
\sum_{i=1}^{2}  \int_{\triangle \~A \~P_i\~B} \dist^p ( x,\pd\om )\d x\notag\\
&\ge \sum_{i=1}^{2} 2^{-p} r_i^p \H^2(\triangle \~A \~P_i\~B)
= \sum_{i=1}^{2} 2^{-p-2} r_i^p \H^2(\triangle A P_i B)\notag\\
&\ge 
\sum_{i=1}^{2} 3^{-p} 2^{-p-3} h_i^{p+1} D
\ge 3^{-p} 2^{-p-3} D \cdot \max_{i=1,2} h_i^{p+1} .
\label{lowe bound on average-distance}
 \end{align}
Recalling that $\H^2(\Omega) \le D(h_1+h_2)$, $ \H^1(\pd \Omega) \ge 2D $,
we get
\[ 
\frac{\H^2(\Omega)^{p+1}}{\H^1(\pd \Omega)^{p}} \le \frac{D^{p+1}(h_1+h_2)^{p+1}}{ (2D)^p}
= 2^{-p} D(h_1+h_2)^{p+1} \le 2 D   \cdot \max_{i=1,2} h_i^{p+1}.\]
Hence \eqref{lowe bound on average-distance} gives
\begin{align*}
 \int_\om \dist^p ( x,\pd\om )\d x&\ge 
C\frac{\H^2(\Omega)^{p+1}}{\H^1(\pd \Omega)^{p}},\qquad C= 3^{-p}  2^{-p-4},
\end{align*}
and \eqref{p claim1} is proven. From \eqref{a1-2} we know that
\[ \frac{\H^1(\pd \Omega_n)}{\H^2(\Omega_n)}  \le C_2 \Lra 
\frac{\H^2(\Omega_n)^{p+1}}{\H^1(\pd \Omega_n)^{p+1}} \ge C_2^{-p-1}, \]
so the above inequality gives
\begin{align*}
\int_\om \dist^p ( x,\pd\om_n )\d x&\ge 
C\frac{\H^2(\Omega_n)^{p+1}}{\H^1(\pd \Omega_n)^{p+1}}\H^1(\pd \Omega_n)
\ge CC_2^{-p-1}\H^1(\pd \Omega_n).
\end{align*}
Now, any minimizing sequence $\{\om_n\}$ is such that, for all sufficiently large $n$,
\[  \E(\om_n) \le \inf \E +1, \]
thus
\begin{align*}
\inf \E+1 & \ge  \E(\om_n) \ge \int_{\om_n} \dist^p ( x,\pd\om_n )\d x
\ge CC_2^{-p-1}\H^1(\pd \Omega_n),
\end{align*}
\eqref{inf E} shows that  $\inf \E<+\8$, completing the proof.

\end{proof}

\begin{remark}\label{remark-new}
{\it We note that it is an interesting geometric question by itself to study what the optimal constant $C$ for the inequality \eqref{p claim1}. Furthermore, one may ask if the form of the inequality is optimal. That is, one may ask, given $\H^2(\Omega)$ and $\H^1(\pd \Omega)$, what is the minimum of  $\int_\Omega \dist^p(x,\pd \Omega)\d x$, which is a constrained optimization problem related to the one considered in this work.
}
\end{remark}

\section{Existence} \label{sec:existence}

In this section we will prove that the $E_{p,\la}$ admits a minimizer in $\A$.
As our arguments rely on a lower semicontinuity result, namely Lemma \ref{basic}
below, we need first to introduce a metric on $\A$.

\bigskip

For any $\Omega_1, \Omega_2 \in {\mathcal {A}}$, define
\begin{eqnarray} \label{d}
d(\Omega_1, \Omega_2) := {\mathcal{H}}^2 (\Omega_1 \triangle \Omega_2),
\end{eqnarray}
where $\triangle$ denotes the symmetric difference.
Set
$$\-{\mathcal A}:= \text{completion of } {\mathcal A} \text{ with respect to } d.$$ 
Before we can proceed, we need to characterize the elements of $\-\A \setminus \A$:
we cannot exclude a priori that an element $\om\in \-\A$ can be quite irregular:
\begin{enumerate}
	\item $\om\in \-\A$ needs not to be closed: indeed it is very possible for a sequence
	of compact sets to converge to an open set in the metric $d$. For instance,
let  $\om_n $ be the closed ball of radius $1-1/n$ centered around the origin, then
it converges to the open ball, centered around the origin, of radius 1.

\item As we do not have any a priori bounds on the diameter
of elements of $\A$, a set
$\om\in \-\A$ needs not to be bounded.  

\item The distance $d$ is insensitive to perturbations on $\H^2$-negligible sets. Therefore,
we cannot exclude that $\-\A$ might contain 
 compact convex sets {\em up to $\H^2$-negligible sets}. Thus whether a generic element
 in $\-\A$ is convex or not is unclear.
\end{enumerate}
In view of the above mentioned issues, we cannot assume neither compactness, nor convexity,
for elements of $\-\A$. Our goal is to show (see Lemma \ref{existence of limit convex set}
below) that minimizing sequences must
converge to some element in $\A$.

\bigskip

The next result, from \cite{Slepcev}, will be crucial for our convergence arguments.

\begin{lemma}
	\label{convergence of curves}
	Consider a sequence of constant speed parameterized curves $\gamma_n :[0,1]\lra K$, where
	$K\sse \R^d$ is some compact set. Assume moreover that
	\begin{align}
	\sup_n L(\gm_n)<+\8,\qquad \sup_n \|\gm_n\|_{BV([0,1];\R^d)} <+\8,  
	\label{bounded length and BV norm}
	\end{align}
	where $\|\cdot\|_{BV([0,1];\R^d)}$ denotes the bounded variation norm. Then there exists a curve $\gamma:[0,1]\lra K$ such that:
	\begin{enumerate}
		\item $\gm_n\to \gm$ in $C^\al([0,1] ;\R^d )$ for all $\al\in [0,1)$,
		\item $\gm_n'\to \gm'$ in $L^p(0,1;\R^d)$ for all $p<+\8$,
		\item $\gm_n''\rhu \gm''$ weakly as measures.
	\end{enumerate}
\end{lemma}

\begin{remark}\label{remark 1}
We remark that this convergence result is quite strong: consider a sequence $\{\om_n\}\sse \A$
and let $\gm_n$ be constant speed parameterizations of $\pd\om_n$. Note that $\gm_n$
are all closed curves. 
Assume that we are under the hypotheses of Lemma \ref{convergence of curves}, hence
there exists $\gm:[0,1]\lra K$ such that $\gm_n\to \gm$ in 
$C^\al([0,1] ;\R^d )$ for all $\al\in [0,1)$. In particular, we can define $\om$
to be the bounded region delimited by the graph of $\gm$, and we have the uniform convergence
of the boundaries, which in turn gives $d_\H (\pd\om_n,\pd\om)\to 0$. Here $d_\H$ denotes
the Hausdorff distance
\[d_\H(X,Y):=\max \bigg\{ \sup_{x\in X}\dist(x,Y) ,\sup_{y\in Y}\dist(y,X) \bigg\}. \]
Such strong convergence also implies that the characteristic functions $\chi_{\om_n}$
converge to $\chi_{\om}$ in $L^p$, $p\in [1,+\8)$, since 
\begin{align*}
\|\chi_{\om_n}-\chi_{\om}\|_{L^p(\R^d
	)}^p \le \H^2 (\om_n \triangle \om  )
\le \max\Big\{ \sup_n L(\gm_n),L(\gm) \Big \} \cdot d_\H (\pd\om_n,\pd\om)\to 0.
\end{align*}
\end{remark}

\begin{lemma}
	Consider a minimizing sequence 
	$\om_n \sse \A$, then there exists $\om\in\A$, and a sequence $x_n \sse \mathbb{R}^n$ such that $\om_n + x_n \to \om$ in the metric $d$.
	\label{existence of limit convex set}
\end{lemma}

Note that, since our energy is translation invariant, the above convergence result is sufficient for our purposes.

\begin{proof}
	In this proof it is more convenient to work with constant speed, 
		instead of arc-length,
	parameterizations.
	
	Consider minimizing sequence 
	$\{\om_n\}\sse \A$, and
	let $\vph_n:[0,1]\lra \pd \om_n$ be constant speed parameterizations. Note all
	$\pd\om_n$ are closed curves, and as $\E$ is translation invariant,
	we can replace $\Omega_n$ with translated copies (which, for brevity, we still denote by $\Omega_n$, and by $\varphi_n$ the parameterization of $\partial \Omega_n$) such that $\varphi_n(0)=\varphi_n(1)=0$.
	We show that we are under the conditions \eqref{bounded length and BV norm}: first, 
	the upper bound on the perimeter \eqref{p-1} and $\Omega_n \subseteq \mathbb{R}^2$ ensures all $\om_n$
	are contained in some compact set $K$.
	As the curves $\vph_n$ are parameterized by constant speed,
	we have $\|\vph_n'\| = L(\vph_n) =\H^1(\pd\om_n)$ a.e. Then, in view of Lemma \ref{p}, we infer
	\eqref{bounded length and BV norm}. Therefore there exists a limit curve
	$\vph:[0,1]\lra K$ such that the convergences in Lemma \ref{convergence of curves} hold.
	Since 
	$\vph_n(0)=\vph_n(1)=0$ for all $n$, we get $\vph(0)=\vph(1)=0$ too. We define $\om$
	to be the {\em bounded} area delimited by $\vph$, and the graph of $\vph$ turn out to be
	$\pd\om$. By construction, $\om$ is compact. 
	
	\medskip
	
	We need to check it is convex:
	consider arbitrary $P,Q\in \om$, $t\in (0,1)$, and we show that $(1-t)P+tQ\in \om$.
	Consider sequences $P_n,Q_n\in \om_n$ such that $P_n\to P$, $Q_n\to Q$: since
	each $\om_n$ is convex, $(1-t)P_n+tQ_n\in \om_n$. By Lemma \ref{convergence of curves},
	we know $\|\vph_n-\vph\|_{C^0([0,1];\R^2)} \to 0$. As a consequence, 
	\[ d_\H (\pd\om_n,\pd\om)\to 0 \]
	too,
	This allows us to choose, for each $n$, another point $z_n\in \om$ such that
	$|z_n- (   (1-t)P_n+tQ_n ) | \le  d_\H (\pd\om_n,\pd\om)$. By construction, now
	the sequences
	$(1-t)P_n+tQ_n$ and $z_n$ have the same limit. As
	$(1-t)P_n+tQ_n \to (1-t)P+tQ$, and $z_n\to z$, hence $z=(1-t)P+tQ$, using the compactness
	of $\om$ finally gives $z\in \om$.

	\medskip
	
	Finally, we check that $\dim_\H\om=2$. Since the ambient space $\R^2$ has already
	Hausdorff dimension two, it suffices to show that $\om$ contains a set of Hausdorff dimension
	two. For each $n$, we can use the construction from the proof of Lemma \ref{p} on each $\om_n$:
	we showed the existence of triangles $T_i:=\triangle \~A\~P_i\~B$
	(see Figure \ref{construction upper bound on perimeter}) whose distance to the boundary is at least
	$r_i/2$, with $r_i$ being the incenter which satisfied $r_i\ge h_i/3$. 
	Now, since we showed in the proof of Lemma \ref{p} that
	\[ \sum_{i=1}^2 h_{i,n} \diam \om_n \ge \H^2(\om_n), \]
	and 
	$$\H^2(\om_n)\ge C_1, \qquad\diam \om_n\le  \H^1( \pd\om_n )\le C_3$$
	due to Corrollary \ref{a1}, we get
	\[ \sum_{i=1}^2 h_{i,n} \ge  \frac{\H^2(\om_n)}{\diam\om_n} \ge \frac{C_1}{C_3}>0.\]
	This shows that at least one of the triangles $T_{i,n}$, $i=1,2$, must be non degenerate
	since its inradius is bounded from below by
	\[  \max_{i=1,2} r_{i,n} \ge \max_{i=1,2} \frac{h_{i,n}}{3} \ge \frac{C_1}{6C_3}, \]
	and the proof is complete.
\end{proof}

Lemma \ref{existence of limit convex set} is of crucial importance: since we are interested
in the minimizers of $\E$, this allows us to reduce the minimization problem to 
$\A$, and neglect the highly irregular elements of $\-\A\setminus\A$.

\begin{lemma}\label{basic}
Given $p\geq 1$, $\la>0$, and a minimizing sequence $\Omega_n\sse {\mathcal A}$ converging to $ \Omega \in {\mathcal A}$ with respect to $d$, 
then it 
holds:
\begin{align}
\H^2(\Omega)&=\lim_{n\to+\8}\H^2(\Omega_n),\label{con1}\\
\H^1(\pd \Omega)&\leq \liminf_{n\to+\8}\H^1(\pd \Omega_n),\label{con2}\\
\int_{\Omega} \dist^p(x,\pd \Omega)\d x&=\lim_{n\to+\8}\int_{\Omega_n} \dist^p(x,\pd \Omega_n)\d x.\label{con3}
\end{align}
\end{lemma}

\begin{proof}
Estimate \eqref{con1} follows from the definition of the metric $d$ and Remark \ref{remark 1}.

To prove \eqref{con2}, recall that 
the perimeter $\H^1(\pd \Omega_n)$ is the total variation of the characteristic function of $\Omega_n$.
Convergence $\Omega_n\to \Omega $ with respect
to $d$ implies (see Remark \ref{remark 1})  
$$ \chi_{\Omega_n}\to \chi_{\Omega} \text{ strongly in } L^1(\R^2),$$
with ``$\chi$'' denoting the characteristic function of the subscribed set.
Thus \eqref{con2} follows from the lower-semicontinuity of the total variation semi-norm.

 To prove \eqref{con3}, note that
\begin{align*}
\int_{\Omega_n} \dist^p(x,\pd \Omega_n)\d x & =\int_{\Omega_n\backslash \Omega} \dist^p(x,\pd \Omega_n)\d x+\int_{\Omega_n\cap \Omega} \dist^p(x,\pd \Omega_n)\d x\\
\int_\Omega \dist^p(x,\pd \Omega)\d x & =\int_{\Omega \backslash \Omega_n} \dist^p(x,\pd \Omega)\d x+\int_{\Omega_n\cap \Omega} \dist^p(x,\pd \Omega)\d x,
\end{align*}
hence
\begin{align}
\bigg|\int_{\Omega_n}& \dist^p(x,\pd \Omega_n)\d x-\int_{\Omega} \dist^p(x,\pd \Omega)\d x \bigg| \notag\\
&\leq \int_{\Omega_n\backslash \Omega} \dist^p(x,\pd \Omega_n)\d x+
\int_{\Omega \backslash \Omega_n} \dist^p(x,\pd \Omega )\d x \label{en-1}\\
&+\int_{\Omega_n\cap \Omega}| \dist^p(x,\pd \Omega_n)-\dist^p(x,\pd \Omega)|\d x.\label{en-2}
\end{align}
 By Lemma \ref{p},
 \[  \diam(\Omega_n) \leq  \H^1(\pd \Omega_n) \leq C_3.  \]
 According to \eqref{con2},
  \[  \diam(\Omega) \leq  \H^1(\pd \Omega) \leq \liminf_{n\to+\8}\H^1(\pd \Omega_n) \leq C_3.  \]
Therefore,
\begin{align*}
\int_{\Omega_n\backslash \Omega} \dist^p(x,\pd \Omega_n)\d x & \leq \H^2(\Omega_n\backslash \Omega)
( \diam(\Omega_n))^p \leq \H^2(\Omega_n\backslash \Omega) C_3^p \to 0,\\
\int_{\Omega \backslash \Omega_n} \dist^p(x,\pd \Omega)\d x &\leq \H^2(\Omega \backslash \Omega_n) (\diam( \Omega))^p
\leq \H^2(\Omega \backslash \Omega_n)  C_3^p \to 0,
\end{align*}
hence the sum in \eqref{en-1} goes to zero. To estimate \eqref{en-2}, 
denote by $d_\H$ the Hausdorff distance, and
note that,
by the Mean Value theorem, it holds
\begin{align*}
\int_{\Omega_n\cap \Omega}&| \dist^p(x,\pd \Omega_n)-\dist^p(x,\pd \Omega)|\d x\\
&\leq \int_{\Omega_n\cap \Omega}| \dist(x,\pd \Omega_n)-\dist(x,\pd \Omega)|\\
&\cdot p \sup_{x\in \Omega_n\cap \Omega}\Big(\max \{\dist(x,\pd \Omega_n),
\dist(x,\pd \Omega)\}\Big)^{p-1}\d x\\
&\leq \H^2(\Omega_n\cap \Omega) d_\H(\pd \Omega_n,\pd \Omega)\cdot p \Big(\max \{\diam \Omega_n,\diam \Omega\}\Big)^{p-1}\\
&\leq  \H^2(\Omega_n\cap \Omega)  \ d_\H(\pd \Omega_n,\pd \Omega)\cdot p \ C_3^{p-1}\to 0.
\end{align*}
Thus the term in \eqref{en-2} goes to zero too, and \eqref{con3} is proven.
\end{proof}

Now we prove part (1) of Theorem \ref{main}, i.e., the existence of minimizers in ${\mathcal A}$. 
\begin{lemma}\label{exist}
For any $p\geq1$, $\la>0$, the functional $\E$ admits a minimizer $\Omega \in {\mathcal A}$, which satisfies:
\begin{align*}
\H^2(\Omega) \geq C_1,\qquad   \H^1(\pd \Omega)&\leq C_3,
\end{align*}
with $C_1$ (resp. $C_3$) defined in \eqref{a1-1} (resp. \eqref{p-1}).
\end{lemma}

\begin{proof}

Corollary \ref{a1} gives
$ \H^2(\Omega_n) \geq C_1$ for any sufficiently large $n$,
 and Lemma \ref{basic} gives
\begin{equation}\label{AreaC0}
\H^2(\Omega)=\lim_{n\to+\8}\H^2(\Omega_n)\geq C_1.
\end{equation}
Based on  Lemma \ref{basic} and equation \eqref{p-1},
\begin{eqnarray} \label{PerLSC}
\H^1(\pd \Omega) \leq C_3.
\end{eqnarray}

Lemma \ref{basic} gives
\begin{equation}\label{penLSC}
\frac{\H^1(\pd \Omega)}{\H^2(\Omega)}\leq \liminf_{n\to+\8}\frac{\H^1(\pd \Omega_n) }{\H^2(\Omega_n)}.
\end{equation}
and
\begin{align}
\int_{\Omega} \dist^p(x,\pd \Omega)\d x=\lim_{n\to+\8}\int_{\Omega_n} \dist^p(x,\pd \Omega_n)\d x.\label{enC0}
\end{align}
 Combining \eqref{penLSC}
and \eqref{enC0} gives
\begin{align*}
\E(\Omega)&=\int_{\Omega} \dist^p(x,\pd \Omega)\d x +\la\frac{\H^1(\pd \Omega)}{\H^2(\Omega)}\\
&\leq \lim_{n\to+\8}\int_{\Omega_n} \dist^p(x,\pd \Omega_n)\d x+  \la \liminf_{n\to+\8}\ \frac{\H^1(\pd \Omega_n) }{\H^2(\Omega_n)} 
\leq\liminf_{n\to+\8}\E(\Omega_n) = \inf_{\-{\mathcal A}} \E,
\end{align*}
hence $\Omega$ is effectively a minimizer of $\E $ in $ \-{\mathcal A}$. Lemma \ref{existence of limit convex set} shows $\Omega \in {\mathcal A}$.

\end{proof}

\section{Regularity} \label{sec:proof}

Now we prove part (2) of Theorem \ref{main}. The proof will be split over Lemmas \ref{C1}
and \ref{C2}.
\begin{lemma}\label{cut}
Let $S$ be a compact, convex set, with Hausdorff dimension equal to 2. Let $w_1$, $w_2\in \pd S$ be arbitrary
distinct points, and let $\sigma$ be the segment with endpoints $w_1$ and $w_2$. Denoting by $S_1$
and $S_2$ the two connected components of $S\backslash \sigma$, then both $S_1$, $S_2$ are convex.
\end{lemma}

\begin{proof}
Endow $\R^2$ with a Cartesian coordinate system. Upon rotation and reflection, assume that $\sigma$
lies in the $y$-axis, and $S_1\sse \{x>0\}$, $S_2\sse \{x<0\}$. Clearly, given points
$u,v\in S_1$, the segment $\xi$ between $u$ and $v$ lies entirely in $S\cap \{x>0\}=S_1$, hence
$S_1$ is convex. The proof for $S_2$ is analogous.
\end{proof}

\begin{lemma}\label{C1}
($C^1$-regularity) For any $p\geq1$, $\la>0$, any minimizer of $\E$ is $C^1$-regular.
\end{lemma}

\begin{proof}
Consider an arbitrary minimizer $\Omega \in {\mathcal A}$. 
Endow $\R^2$ with a polar coordinate system.
We parameterize $\pd \Omega$
by a closed Lipschitz curve
\begin{equation*}
\gm:[0,2\pi]\lra \pd \Omega.
\end{equation*}
The proof is achieved by a contradiction argument. Assume that $\Omega$ is not $C^1$-regular.
That is, $\gm$ is not $C^1$-regular at some point $ t_0$. 
Upon rotating the coordinates, we can also assume $t_0\in (0,2\pi)$.
Since $\Omega$ is convex,
both one-sided derivatives 
\begin{equation*}
l^-:=\lim_{t \to t_0^-}\gm'(t),\qquad l^+:=\lim_{t \to t_0^+}\gm'(t)
\end{equation*}
are well-defined \cite{AFP, LS}. Denote by $\al$ the angle between $l^-$ and $l^+$. Clearly, $\al\neq \pi$.
\begin{figure}[ht]
\begin{center}
\includegraphics[scale=0.6]{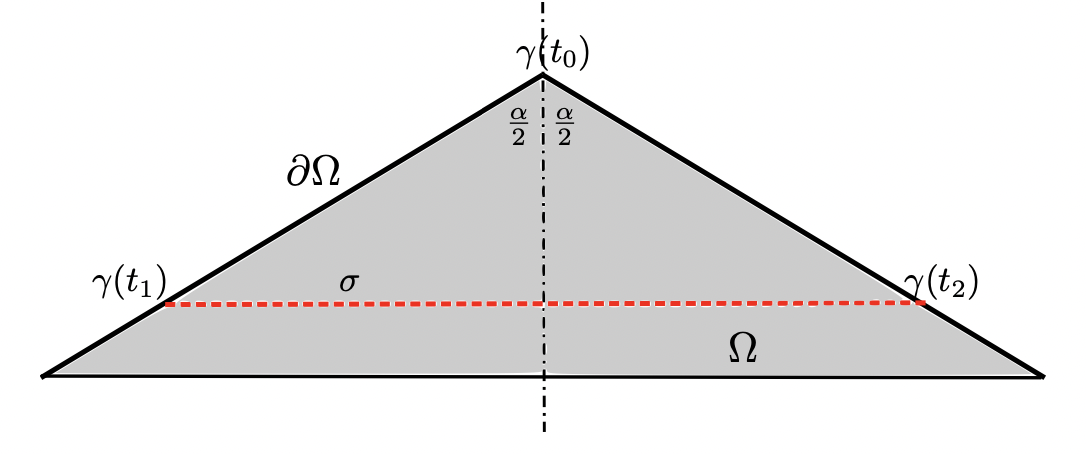}
\caption{\small{A schematic representation (near $\gm(t_0)$, in first order approximation in $\vep$) of the construction of $\Omega_\vep$.}}
\label{corner}
\end{center}
\end{figure}

Figure \ref{corner} is a representation (in first order approximation) of $\pd \Omega$
near $\gm(t_0)$.
For small parameters $0<\vep\ll1$, construct the competitor $\Omega_\vep$ as follows:
\begin{enumerate}
\item  Choose
 $t_1<t_0<t_2$ such that (in first order approximation in $\vep$)
 \begin{equation*}
\H^1(\gm([t_1,t_0]))=\H^1(\gm([t_0,t_2]))=\vep+O(\vep^2).
\end{equation*}

 \item Denote by 
 $$\sigma:=\{(1-s)\gm(t_1)+s\gm(t_2):s\in [0,1]\}$$ 
 the line segment between $\gm(t_1)$ and $\gm(t_2)$, and set
 \begin{equation}\label{comp}
L:=\big( \pd \Omega \backslash \gm([t_1,t_2])\big) \cup \sigma.
\end{equation}
Note that such $L$ is a convex Jordan curve, 
and denote by $\Omega_\vep$ the bounded region delimited by $L$.
\end{enumerate}
By construction, in first order approximation in $\vep$, it holds
\begin{align}
\H^1(\pd \Omega_\vep)&=\H^1(\pd \Omega)-2\vep(1- \sin(\al/2)) +O(\vep^2),\label{Dp}\\
\H^2(\Omega_\vep)&=\H^2(\Omega)-\frac{\vep^2\sin\al}2 +o(\vep^2)=\H^2(\Omega) +O(\vep^2).\label{Da}
\end{align}
Moreover, it is straightforward to show that
\begin{equation}\label{De}
\int_{\Omega_\vep} \dist^p(x,\pd \Omega_\vep)\d x\leq\int_{\Omega} \dist^p(x,\pd \Omega)\d x. 
\end{equation}
%

Recalling that $\H^2(\Omega)>0$ (since $\Omega$ is a minimizer), 
combining \eqref{Dp}, \eqref{Da} and \eqref{De} gives (in first order approximation in $\vep$)
\begin{align*}
\E(\Omega_\vep)&= \int_{\Omega_\vep} \dist^p(x,\pd \Omega_\vep) \ d x + \la \frac{\H^1(\pd \Omega_\vep)}{\H^2(\Omega_\vep)}\\
&\leq \int_{\Omega } \dist^p(x,\pd \Omega )\ d x + \la \frac{\H^1(\pd \Omega)-2\vep(1- \sin(\al/2)) +O(\vep^2)}{\H^2(\Omega) +O(\vep^2)}\\
&=\int_{\Omega} \dist^p(x,\pd \Omega)\d x + \la \frac{\H^1(\pd \Omega)}{\H^2(\Omega) } 
-\frac{2\la\vep(1- \sin(\al/2))}{\H^2(\Omega) } +O(\vep^2)\\
&=\E(\Omega)-\frac{2\la\vep(1- \sin(\al/2))}{\H^2(\Omega) } +O(\vep^2)\\
&=\min_{\mathcal A}\E-\frac{2\la\vep(1- \sin(\al/2))}{\H^2(\Omega) } +O(\vep^2),
\end{align*}
which is a contradiction for sufficiently small $\vep$. Thus $\Omega$ must be $C^1$-regular.
\end{proof}

\begin{lemma}\label{C2}
Given $p\geq1$, $\la>0$, a minimizer $\Omega$ of $\E$, let $\gm:[0,\H^1(\pd \Omega)]\lra \pd \Omega$ be 
an arc-length parameterization. Then it holds
 \begin{align}
\limsup_{h\to0}&\frac{|\gm(t+2h)-2\gm(t)+\gm(t-2h)|}{h^2} \leq  4C,\label{C11}
\end{align}
for any $t$, where $C$ is some constant depending only on $\la$ and $p$ (and independent of $\Omega$). 
\end{lemma}
We remark that \eqref{C11} implies $C^{1,1}$-regularity of $\pd \Omega$.

\begin{proof}
Consider an arbitrary point $p_0\in \pd \Omega$.
Since we proved that $\Omega$ is $C^1$-regular, consider a (local) orthogonal coordinate system with
origin in $p_0$, and $x$-axis oriented along the tangent derivative (at $p_0$), such that
$\Omega$ is entirely contained in the half-plane $\{y\geq 0\}$. 
The boundary $\pd \Omega$ is thus (locally) the graph of some nonnegative function $f$. Clearly, 
such $f$ satisfies $f(0)=0$.
\begin{figure}[ht]
\begin{center}
\includegraphics[scale=0.6]{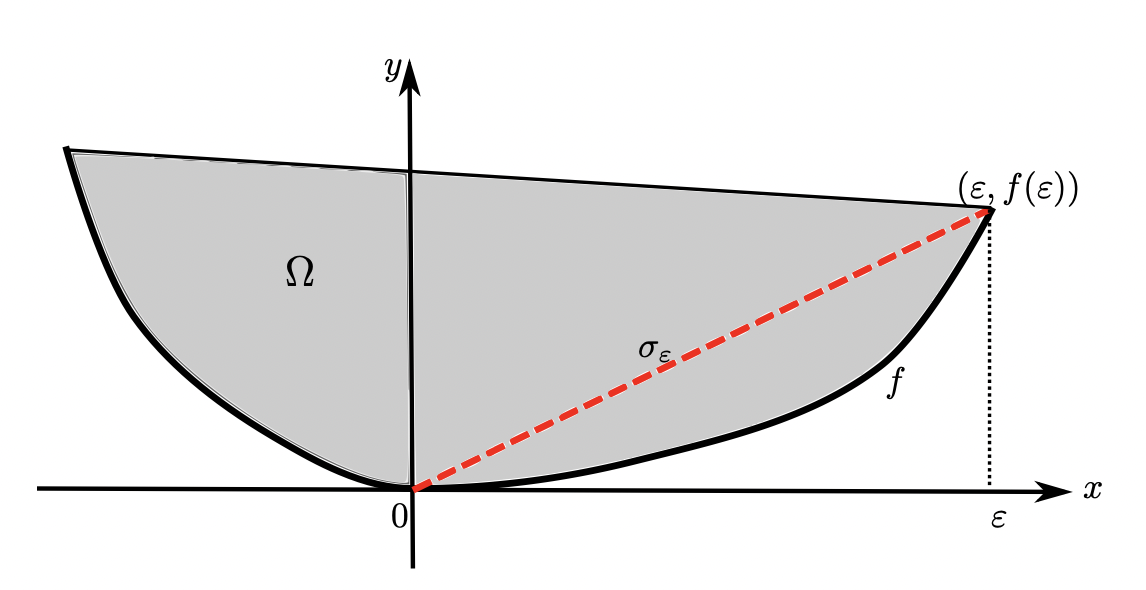}
\caption{\small{A schematic representation of the construction near $p_0=(0,0)$.}}
\label{smooth}
\end{center}
\end{figure}

Choose an arbitrary $0<\vep\ll 1$. 
Denote by
$$\sigma_\vep:=\{(x,y):0\leq x\leq \vep,\ y=x\cdot f(\vep)/\vep\} $$
the segment between the origin and $(\vep,f(\vep))$.
Let $L_\vep$ be the curve obtained by replacing $f([0,\vep])$ with $\sigma_\vep$. That is, 
\begin{equation*}
L_\vep := (\pd \Omega \backslash f([0,\vep])) \cup \sigma_\vep.
\end{equation*}
By construction (see Lemma \ref{cut}) $L_\vep$ is a convex Jordan curve, and let
$\Omega_\vep$ be the bounded region delimited by $ L_\vep$. Note that:
\begin{enumerate}

\item Clearly we can infer
\begin{equation}\label{De2}
\int_{\Omega_\vep} \dist^p(x,\pd \Omega_\vep)\d x\leq\int_{\Omega} \dist^p(x,\pd \Omega)\d x. 
\end{equation}

\item For areas, since by construction it holds $\Omega_\vep\sse \Omega$, we have 
\begin{align}
\H^2(\Omega)-\H^2(\Omega_\vep) &=\H^2(\Omega \backslash \Omega_\vep)\notag\\
&=\H^2(\{ (x,y):0\leq x\leq \vep,\ f(x)\leq y\leq x\cdot f(\vep)/\vep\})\notag\\
&=\int_0^\vep [xf(\vep)/\vep - f(x)] \d x
=\frac{f(\vep)\vep}{2} -\int_0^\vep f(x)\d x.\label{Da2}
\end{align}

\item For perimeters, note that $f'(0)=0$, so $|f'|$ is small near $0$. In particular,
by choosing sufficiently small $\vep$, we can ensure that
$$\sqrt{1+|f'(x)|^2}\leq2 \qquad \text{for all } x\in(0,\vep),$$
and also
$|f(\vep)|/ \vep$ can be made as small as we need, so to satisfy
\[  \sqrt{1+\frac{f(\vep)^2}{\vep^2}} 
=1+\frac{f(\vep)^2}{2\vep^2}  -\frac{1}{8}\Big(\frac{f(\vep)^2}{\vep^2}\Big)^2 + O\Big(
\Big(\frac{f(\vep)^2}{\vep^2} \Big)^3 \Big)\le
 1+\frac{f(\vep)^2}{2\vep^2}. \]
Therefore,
\begin{align*}
\H^1(\pd \Omega)-\H^1(\pd \Omega_\vep) 
&=\int_0^\vep \Big(\sqrt{1+|f'(x)|^2} - \sqrt{1+\frac{f(\vep)^2}{\vep^2}
}\Big)\d x \\
&\ge 
\int_0^\vep \Big(\sqrt{1+|f'(x)|^2} -  1-\frac{f(\vep)^2}{2\vep^2}\Big)\d x ,
\end{align*}
where, since for sufficiently small
 $\vep\ll1$ the quantity $\frac{f(\vep)^2}{2\vep^2}$ can be made arbitrarily
small, we have
\[  \int_0^\vep \frac{f(\vep)^2}{2\vep^2}\d x =o(\vep),\qquad \vep\ll1.\]
 Thus, for all sufficiently small $\vep$,
\begin{align}
\H^1(\pd \Omega)-\H^1(\pd \Omega_\vep) 
&=\int_0^\vep \Big(\sqrt{1+|f'(x)|^2} - \sqrt{1+\frac{f(\vep)^2}{\vep^2}
}\Big)\d x \notag\\
&\ge\int_0^\vep \Big(\sqrt{1+|f'(x)|^2} -1\Big)\d x + o(\vep) \notag\\
&=\int_0^\vep \frac{|f'(x)|^2}{\sqrt{1+|f'(x)|^2} +1}\d x  + o(\vep) \geq
\frac13\int_0^\vep |f'(x)|^2\d x.\label{Dp2}
\end{align}
\end{enumerate}
Combining \eqref{De2}, \eqref{Da2} and \eqref{Dp2} gives
\begin{align}
\E(\Omega_\vep)&=\int_{\Omega_\vep} \dist^p(x,\pd \Omega_\vep)\d x+\la \frac{\H^1(\pd \Omega_\vep)}{\H^2(\Omega_\vep)}\notag\\
&\leq
\int_{\Omega} \dist^p(x,\pd \Omega)\d x + \la \frac{\H^1(\pd \Omega)-\frac13\int_0^\vep |f'(x)|^2\d x}{\H^2(\Omega)
-(\frac{f(\vep)\vep}{2} -\int_0^\vep f(x)\d x)}.\label{E1}
\end{align}
Since
\begin{align*}
\frac{\H^1(\pd \Omega)}{\H^2(\Omega)
-(\frac{f(\vep)\vep}{2} -\int_0^\vep f(x)\d x)} &=
\frac{\H^1(\pd \Omega)}{\H^2(\Omega)}\cdot\frac{\H^2(\Omega)}{
\H^2(\Omega)-(\frac{f(\vep)\vep}{2} -\int_0^\vep f(x)\d x)}\\
&=
\frac{\H^1(\pd \Omega)}{\H^2(\Omega)}\cdot\bigg(1+\frac{\frac{f(\vep)\vep}{2} -\int_0^\vep f(x)\d x}{
\H^2(\Omega)-(\frac{f(\vep)\vep}{2} -\int_0^\vep f(x)\d x)}\bigg),
\end{align*}
estimate \eqref{E1} reads
\begin{align}
\E(\Omega_\vep)&\leq\int_{\Omega} \dist^p(x,\pd \Omega)\d x -\la \frac{\frac13\int_0^\vep |f'(x)|^2\d x}{\H^2(\Omega)
-(\frac{f(\vep)\vep}{2} -\int_0^\vep f(x)\d x)}\notag\\
&+\la\frac{\H^1(\pd \Omega)}{\H^2(\Omega)}\cdot\bigg(1+\frac{ \frac{f(\vep)\vep}{2} -\int_0^\vep f(x)\d x}{
\H^2(\Omega)-(\frac{f(\vep)\vep}{2} -\int_0^\vep f(x)\d x)}\bigg)\notag\\
&=\E(\Omega)+\la\frac{\frac{\H^1(\pd \Omega)}{\H^2(\Omega)}(\frac{f(\vep)\vep}{2} -\int_0^\vep f(x)\d x) - \frac13\int_0^\vep |f'(x)|^2\d x}{
\H^2(\Omega)-(\frac{f(\vep)\vep}{2} -\int_0^\vep f(x)\d x)}. \label{E2}
\end{align}
Since $\Omega$ is a minimizer, Lemma \ref{exist} gives $\H^2(\Omega)>0$, and note that
$$\frac{f(\vep)\vep}{2} -\int_0^\vep f(x)\d x\leq \frac{\H^2(\Omega)}2$$
for all sufficiently small $\vep$, hence the denominator in \eqref{E2}
is positive. Thus the minimality of $\Omega$ forces the numerator in \eqref{E2} to be nonnegative, i.e.,
\begin{equation}
3\frac{\H^1(\pd \Omega)}{\H^2(\Omega)}\bigg(\frac{f(\vep)\vep}{2} -\int_0^\vep f(x)\d x\bigg) - \int_0^\vep |f'(x)|^2\d x\geq 0.\label{min}
\end{equation}
Equation \eqref{Da2} shows
\[ 
 \frac{f(\vep)\vep}{2} -\int_0^\vep f(x)\d x = \H^2(\Omega)-\H^2(\Omega_\vep) \ge 0 
\]
since by construction $\Omega_\vep\sse \Omega$.
Lemma \ref{exist} gives
\begin{equation*}
 \H^2(\Omega)\geq C_1,\qquad  \H^1(\pd \Omega)\leq C_3,
\end{equation*}
hence
\begin{equation*}
3\frac{\H^1(\pd \Omega)}{\H^2(\Omega)} \leq \frac{3 C_3 }{C_1}=:C,
\end{equation*}
and \eqref{min} forces
\begin{equation}\label{E3}
C \bigg(\frac{f(\vep)\vep}{2} -\int_0^\vep f(x)\d x\bigg) \geq \int_0^\vep |f'(x)|^2\d x.
\end{equation}
Since $\Omega$ is convex, and we assumed (at the beginning of this proof) that $\Omega \sse \{y\geq 0\}$,
$f$ is nonnegative, hence \eqref{E3} forces
\begin{equation}\label{E4}
\frac{C}{2}f(\vep)\vep \geq \int_0^\vep |f'(x)|^2\d x.
\end{equation}
Note that, since $f(0)=0$, it follows
\begin{equation}
f(\vep) = \int_0^\vep f'(x)\d x \leq\int_0^\vep |f'(x)|\d x.\label{f1}
\end{equation}
By H\"older's inequality,
\begin{equation}
\int_0^\vep |f'(x)|^2\d x \geq \frac1\vep\bigg(\int_0^\vep |f'(x)|\d x\bigg)^2,\label{fH}
\end{equation}
hence  
\begin{align*}
\frac{C}{2}&\vep\int_0^\vep |f'(x)|\d x 
\overset{\eqref{f1}}\geq 
\frac{C}{2}f(\vep)\vep
\overset{\eqref{E4}}\geq\int_0^\vep |f'(x)|^2\d x\overset{\eqref{fH}}\geq
\frac1\vep\bigg(\int_0^\vep |f'(x)|\d x\bigg)^2\\
&\Lra \frac{C}{2}\vep^2\geq \int_0^\vep |f'(x)|\d x\geq \int_0^\vep f'(x)\d x =f(\vep).
\end{align*}
The above arguments can be repeated for $\vep<0$, $|\vep|\ll1$ (or equivalently, when the orientation of $x$-axis is inverted).
The arbitrariness of $\vep$ then gives
\begin{align*}
\limsup_{\vep\to0}&\frac{|f(\vep)-2f(0)+f(-\vep)|}{(\vep/2)^2} \leq  4C,
\end{align*}
concluding the proof.
\end{proof}

Now we prove part (3) of Theorem \ref{main}.
\begin{lemma}
Given $p\geq1$, $\la>0$, any minimizer $\Omega$ of $\E$ satisfies
\begin{equation*}
\frac{\H^1(\pd \Omega)}{\H^2(\Omega)} = \frac{p+2}{\la(p+3)}\min_{\mathcal A} \E.
\end{equation*}
\end{lemma}

\begin{proof}
Let $\Omega$ be an arbitrary minimizer. Endow $\R^2$ with a Cartesian coordinate system,
 and assume without loss of generality that $(0,0)$ is in the interior part of $\Omega$.
For any $r>0$, denote by
\begin{equation*}
T_r:\R^2\lra \R^2,\qquad T_r(x):=rx
\end{equation*}
 the homothety of center $(0,0)$ and ratio $r$. 
Note that $T_r(\Omega)\in {\mathcal A}$ for any $r>0$, and the scalings are
\begin{align*}
\int_{T_r(\Omega)} \dist^p(x,\pd T_r(\Omega))\d x&= r^{p+2}\int_{\Omega} \dist^p(x,\pd \Omega)\d x,\\
\frac{\H^1(\pd T_r(\Omega))}{\H^2(T_r(\Omega))}&=\frac1r\cdot \frac{\H^1(\pd \Omega)}{\H^2(\Omega)}.
\end{align*}
Define the function
\begin{equation*}
f:(0,+\8)\lra(0,+\8),\qquad f(r):=\E(T_r(\Omega))=r^{p+2}\int_{\Omega} \dist^p(x,\pd \Omega)\d x+\frac\la r \cdot\frac{\H^1(\pd \Omega)}{\H^2(\Omega)}.
\end{equation*}
Since $f$ is smooth, and attains a global minimum at $r=1$, it follows
\begin{align*}
f'(1)&=(p+2)\int_{\Omega} \dist^p(x,\pd \Omega)\d x-\la \cdot\frac{\H^1(\pd \Omega)}{\H^2(\Omega)}=0\\
&\Lra \int_{\Omega}\dist^p(x,\pd \Omega)\d x=\frac\la{p+2}\cdot\frac{\H^1(\pd \Omega)}{\H^2(\Omega)},
\end{align*}
hence
\begin{equation*}
\E(\Omega)=\frac{\la(p+3)}{p+2}\cdot\frac{\H^1(\pd \Omega)}{\H^2(\Omega)}=\min_{\mathcal A} \E,
\end{equation*}
and the proof is complete.
\end{proof}

%
%

Let us conclude the paper with some final remarks. 
In this paper we investigated the minimization problem for the average distance functional, with perimeter-to-area ratio penalization, in the plane. We proved the existence and $C^{1,1}$ regularity of minimizers, mainly relying on constructing suitable competitors. Echoing and  developing former studies that exclusively focused on either the 1D average distance problem or purely surface area-to-volume ratio question, by considering optimal sets of combined energy from broader and more eclectic perspectives, this study enriches and deepens our understanding of penalized average distance problem. 

\medskip

We remark that 
all the main results of this paper, i.e. bounds
	on the perimeter and area, and $C^{1,1}$-regular of minimizers, 
		 can be also proven if we replace the 
perimeter-to-area term with 		 
		  a generalized ratio 
		 of the form
$	 \la \frac{\H^1(\pd \Omega)^\al}{\H^2(\Omega)^\bt}$,
	symbolizing a perimeter term normalized (by area) with different scaling exponents $\alpha$ and $\beta$.
That is, we consider 
an energy of the form
	\begin{equation}
	\label{generalized energy}
	 E_{p,\la}^{\al,\bt}(\om):= \int_{\Omega} \dist^p(x,\pd \Omega)\d x+
	 \la \frac{\H^1(\pd \Omega)^\al}{\H^2(\Omega)^\bt},
	 	\end{equation}
	where $\al,\bt$ are given powers satisfying $2\bt>\al > \frac{p}{p+1}\bt>0$.
	This last bound, combined with Young's inequality, allows
	us to easily bound the perimeter, and the subsequent results. It can also be quickly checked
	that if $\al>2\bt$, then minimizers are just single points.
One more remark is that according to \eqref{p claim1}, if in \eqref{generalized energy} we pick $\alpha = p, \beta = p+1$ and $\lambda = C $ as in \eqref{p claim1}, we get
\begin{eqnarray}
 E_{p,\la}^{p, p+1 }(\om) \geq  C \frac{\H^2(\Omega)^{p+1}}{\H^1(\pd \Omega)^p} + C\frac{\H^1(\pd \Omega)^p}{\H^2(\Omega)^{p+1}} \geq 2 C. \nonumber
\end{eqnarray}
So in this case if the optimal constant in \eqref{p claim1} is obtained by a circle, the optimal shape for \eqref{func} is a circle. An interesting question worthy further consideration is if the circle would be the minimizer for other parameters, as in similar discussions given in 
\cite{mk,gnru,gnro,ffmmm}.  Another natural question is to ask if in general one may improve the $C^{1,1}$ regularity by  combining the established results with elliptic regularity theory, given that the variation of the perimeter-to-area ratio leads to a system of second order differential equations of the boundary parametrization.

\medskip

In addition, it is interesting to improve the results of this paper to higher dimensions,
again with a generalized ratio penalization. However, the geometric
complexity of higher dimensional objects can increase
significantly, 
and more work is required to exclude more complicated sets (e.g., ``tentacles''), which were not an issue in the planar case,
thus we expected to rely on rather different tools and arguments.


\section*{Acknowledgments}
The work of QD is supported in part by the National Science Foundation under award DMS-2012562 and the ARO MURI Grant W911NF-15-1-0562. 
XYL acknowledges the support of his NSERC Discovery Grant 
``Regularity of minimizers and pattern formation in geometric minimization problems''.

\end{document}